\documentclass[12pt]{amsart}

\newtheorem{thm}{Theorem}[section]

\newtheorem{qu}[thm]{Question}

\theoremstyle{definition}

\theoremstyle{remark}

\begin{document}

\title{The Existence of Minimal Logarithmic Signatures for some Finite Simple Unitary Groups }

\author{\bf A. R. Rahimipour and A. R. Ashrafi$^\star$}

\footnote{$^\star$Corresponding author (Email: ashrafi@kashanu.ac.ir).}

\address{Department of Pure Mathematics, Faculty of Mathematical
Sciences, University of Kashan, Kashan 87317$-$51116, I. R. Iran}

\maketitle
\begin{abstract}
The $MLS$ conjecture states that
every finite simple group has a minimal logarithmic signature. The aim of this paper is
proving the existence of a minimal logarithmic signature for  some simple unitary groups $PSU_{n}(q)$. We report a gap in the proof of the main result of [H. Hong, L. Wang, Y. Yang, Minimal logarithmic signatures for the unitary
group $U_n(q)$, \textit{Des. Codes Cryptogr.} \textbf{77} (1) (2015) 179--191] and present a new proof in some special cases of this result. As a consequence,  the $MLS$ conjecture is still open.

\vskip 3mm


\vskip 3mm

\end{abstract}

\bigskip

\section{Introduction}
In literature, there are several cryptosystems based on the properties of large abelian groups \cite{35}, but applications of non-abelian groups in constructing cryptosystems was started by publishing some pioneering papers of Magliveras and his co-authors \cite{19, 20, 21, 22, 23}. In \cite{19}, a symmetric key cryptosystem based on logarithmic signatures for finite non-abelian permutation groups was proposed and its algebraic properties were studied in \cite{23}. In \cite{24},  two possible approaches for constructing of new public key cryptosystems using group
factorizations were described.

Throughout this paper all groups are assumed to be finite. Suppose $G$ is such a group. A logarithmic signature (LS) for $G$ is an ordered tuple $\alpha = [A_1, \ldots, A_s]$ of subsets $A_i$ of $G$ such that any  element of $G$ can be uniquely decomposed into a product in the form $a_1a_2\cdots a_s$, where $a_i \in A_i$. The concept of logarithmic signature was first defined by  Magliveras and the existence of this type of factorizations  play an important role in secret  cryptosystems such as $PGM$ and  public key cryptosystems like  $MST_1$, $MST_2$ and $MST_3$ \cite{15,20}.

Suppose $G$ is a finite group with a logarithmic signature $\alpha(G) = [A_1, A_2, \cdots, A_s]$. The subsets $A_i$, $1 \leq i \leq s$, in $\alpha$ are called the blocks, and the vector $(r_1, \ldots, r_s)$, $r_i = |A_i|$, is the type of $\alpha$. The logarithmic signature $\alpha$ is said to be nontrivial if $s \geq 2$ and $r_i \geq 2$, for $1 \leq i \leq s$. The length of $\alpha$  is defined as $l(\alpha) = \sum_{i=1}^s|A_i|$. To apply effectively the logarithmic signature in practical problems, we have to define the logarithmic signature  with shortest length. Suppose $|G| = p_1^{\beta_1}\cdots p_k^{\beta_k}$, where $p_i$'s are primes and $\beta_i$'s are positive integers.  An observation by Gonz$\acute{\rm a}$lez Vasco and Steinwandt \cite{6}  shows that $l(\alpha) \geq \sum_{j=1}^k \beta_jp_j$ and so the logarithmic signature $\alpha$ is said to be minimal if $l(\alpha) = \sum_{j=1}^k \beta_jp_j$. A minimal logarithmic signature will be  abbreviated by $MLS$. It is a well-known  conjecture that every finite group has an $MLS$, which is still not proved in general. In \cite[Table 1]{28}, we presented a complete
history of MLS conjecture until now. We encourage the interested reader to consult
the mentioned paper and references therein for more information on this conjecture.

Before we can proceed further, we will need some basic definitions. We refer the interested readers to consult \cite{37} for more details. Suppose $V$ is a finite dimensional vector space over the field $K = GF(q^2)$, where $q$ is a prime power. A mapping $f: V \times V \longrightarrow K$ is said to be conjugate-symmetric sesquilinear form if for every $u, v, w \in V$ and $\lambda \in K$,
\begin{enumerate}
\item $f(\lambda u + v, w) = \lambda f(u,w) + f(v,w)$,

\item $f(u,v) = f(v,u)^q$.
\end{enumerate}

A result in \cite{8} states that the function $f : V \times V \longrightarrow GF(q^2)$ given by $$f(x,y) = tr_{\frac{GF(q^{2n})}{GF(q^2)}}xy^q = \sum_{i=0}^{n-1}(xy^q)^{(q^{2i})}$$ is a non-singular conjugate-symmetric sesquilinear form.

If $f$ is a conjugate-symmetric sesquilinear form then $f(v,v)$ is called the norm of $v \in V$ and we say that the vectors $u$ and $v$ are perpendicular or orthogonal, if $f(u,v) = 0$.
A non-zero vector which is perpendicular to itself is called isotropic. Suppose $W \leq V$. Define $W^\bot =\{ v \in V  | \ f(v,w) = 0; \ \forall w\in W   \}$ and $V^\bot$ is denoted by $rad(f)$. If $W \leq W^\bot$ then $W$ is called a totally isotropic subspace of $V$. A conjugate-symmetric sesquilinear form $f$ is called non-singular, if $rad(f)  = \{ 0\}$. An element $g \in GL(V)$ is called
isometry, if $f(g(u),g(v)) = f(u,v)$, for each $u,v \in V$.  The set of all isometries of $V$ is denoted by $GU_n(q)$. Moreover, the groups $SU_n(q) = SL_n(q^2) \cap GU_n(q)$ and $PSU_n(q) = \frac{SU_n(q)}{Z(SU_n(q))}$ are called the special unitary and projective special unitary groups, respectively.

A Singer cycle is an element $x$ in $GL_n(q)$ of order $q^n - 1$. If $G \leq GL_n(q)$ then the cyclic subgroup $S \cap G$, where $S = \langle x \rangle$ is called a Singer cycle subgroup of $G$. The image of Singer cycle subgroup of a group $G \leq GL_n(q)$ under canonical homomorphism is called a Singer subgroup in projective groups.

\section{Preliminary Results}
In this section our basic notions together with some known results are presented which are crucial throughout this paper.

Suppose $V$ is an $n-$dimensional vector space with projective space $P(V)$. If $W \leq V$ then by definition $P(W)$ is the corresponding subspace in $P(V)$. Set $S = \{ W_i \ | \ 1 \leq i \leq t\}$, where $W_i$ is an $r-$dimensional subspaces of $V$. If for each $i, j$ with $i \ne j$ we have $W_i \cap W_j = \langle 0 \rangle$, then the set $S$ is called an $r-$partial spread of $V$.  An $r-$spread is an $r-$partial spread with this property that $V = \bigcup_{i=1}^tW_i$. Note that for every $r-$partial spread or $r-$spread of $V$ the set $\widetilde{S} = \{ P(W) \ | \ W \in S\}$ is an  $r-$partial spread or $r-$spread of $P(V)$, respectively. Consider the special case of $V = GF(q^{2n})$ as a vector space of dimension $2n$ on $GF(q)$. We also assume that $\alpha$ is a primitive element of $V$ and $W = GF(q^n)$ is a subspace of $V$. If $Wx = \{ wx \ | \ w \in W\}$, $x \in V$, then it is easy to see that $$ Wx \cap Wy = \left\{ \begin{array}{ll} \langle 0 \rangle & y \notin Wx\\ Wx & y \in Wx \end{array} \right..$$
Hence, the set of all $Wx$, $x \in V$, is an $n-$spread for $V$ which is called the classical spread of $V$. This spread is obtained  from the set of all cosets of a Singer  cycle subgroup of order $q^n - 1$ in $V = GF(q^{2n})$. This proves that the classical spread is equal to $S = \{ W_i \ | \ 0 \leq i \leq q^n  \}$, where $W_i$ = $W\alpha^{i(q^n - 1)}$ = $\{ w\alpha^{i(q^n - 1)} \ | \  w \in W  \}$ and $0 \leq i \leq q^n $. Note that $W_0 = W = GF(q^n)$.

Let the group $G$ act on a set $X$. The action is called sharply transitive if for each $x, y \in X$ there exists a unique element $g \in G$ such that $y = gx$. Suppose $A \leq G$, $Y \subseteq X$ and $x \in Y$. We say that $A$ is sharply transitive on $Y$ with respect to $x$, if for each $y \in Y$ there exists a unique $a \in A$ such that $ax = y$.

If $V$ is a vector space and $s \in V$ then the map $T_s : V \longrightarrow V$ is defined as $T_s(v) = sv$. Suppose $\alpha$ is a primitive element of $V = GF(q^{2n})$ and $X \in GL_{2n}(q)$ is the corresponding matrix of $T_\alpha$. Then $C = \langle X \rangle$ is a Singer cycle subgroup. By construction of classical spread $S$, one can easily prove that the subgroup $A =\langle X^{q^n - 1}\rangle $  of order $q^n + 1$ is sharply transitive on $S$ with respect to $W_0 = GF(q^n)$. Moreover, $A$ is sharply transitive on $\widetilde{S}$ with respect to $P(W_0)$.

\begin{thm}\label{thm1} {\rm (}See \cite{37}{\rm )}
If $V$ is an $n-$dimensional vector space over the finite field $K = GF(q^2)$ and $f: V \times V \longrightarrow K$ is a conjugate-symmetric sesquilinear form. Then, the following hold:
\begin{enumerate}
\item If $n = 2m + 1$ then $V$ has a basis $\{ e_1,e_2, \ldots, e_m,w,f_1,f_2, \ldots, f_m\}$ such that
\begin{enumerate}
\item $f(e_i,e_j) = f(f_i,f_j) = 0$, $1 \leq i,j \leq m$;

\item $f(e_i,f_j) = -f(f_j,e_i) = \delta_{ij}$, $1 \leq i,j \leq m$;

\item $f(w,w) = 1$ and $f(e_i,w) = f(f_i,w) = 0$, $1 \leq i \leq m$.
\end{enumerate}
\item If $n = 2m$ then $V$ has a basis $\{ e_1, e_2, \ldots, e_m, f_1, f_2, \ldots, f_m\}$ which satisfies conditions 1$($a$)$ and 1$($b$)$.
\end{enumerate}
\end{thm}

The Witt's lemma \cite[Theorem 3.3]{37} states that if $(V, f)$ and $(W, g)$ are isometric spaces, with $f$ and $g$ nonsingular, and  conjugate-symmetric sesquilinear forms then any isometry  between subspaces
$X$ of $V$ and $Y$ of $W$ can be extended to an isometry of $V$ with $W$. Suppose $P(V)$ is the set of all one dimensional subspaces of $V$. If $v$ is an isotropic vector of $V$ then $\langle v\rangle$ is called an isotropic point. By Witt's lemma, the groups $SU_n(q)$ and $PSU_n(q)$ have transitive actions on isotropic points of $P(V)$ and by \cite[Section 3.6]{37}, the number of isotropic points is $\frac{(q^n - (-1)^n)(q^{n-1} + (-1)^n)}{q^2-1}$.

In the next section,  some subgroups of $SU_n(q)$ and $PSU_n(q)$ for constructing $MLS$ are needed.  The construction of these subgroups are based on some results in \cite{2,12}. It is well-known that the symplectic groups, unitary groups of odd dimension and orthogonal groups of minus type have a Singer cycle subgroup.  The order of a Singer cycle subgroup in the general, special and projective special unitary groups $GU_n(q)$, $SU_n(q)$ and $PSU_n(q)$ are equal to $q^n + 1$, $\frac{q^n + 1}{q+1}$ and $\frac{q^n+1}{d(q+1)}$, respectively. Here $d = (n+1,q)$.

Babai et al. \cite[Table VI]{2} proved that the group $SU_n(q)$, $n$ is odd, has a cyclic subgroup of order $q^{n-1} - 1$ with generator

$$\left[\begin{array}{cc} c & 0 \\ 0 & (c^\ast)^{-1}
\end{array} \right],$$
where $c$ is a generator of a Singer cycle subgroup of $GL_{\frac{n-1}{2}}(q^2)$ and $c^\ast$ denotes the image of the transpose of $c$ under automorphism $a \mapsto a^q$.

If $H \unlhd G$ then the canonical homomorphism from $G$ onto $\frac{G}{H}$ is denoted by $\eta$. The following two results from Singhi and Singhi  \cite{31} were crucial in the  proof of the main result of \cite{11}.

\begin{thm} \label{thm2}
Let $H$ be a normal subgroup of $G$ and $A \subseteq G$ has the property that if $a, b \in A$ and $a \ne b$ then $aH \ne bH$. Let $A^\prime = \eta(A)$ and $[A_1,A_2, \cdots, A_k]$ be an $LS$ for $A$. Define $B_i = \eta(A_i)$, $1 \leq i \leq k$. Then  $[B_1,B_2, \cdots, B_k]$ is an $LS$ for $A^\prime$.
\end{thm}

\begin{thm} \label{thm3}
Suppose $V$ is a finite dimensional vector space over $GF(q)$, $f$ is a non-singular conjugate-symmetric sesquilinear form. Suppose  $L$  is one of the following subsets of $P(V)$:
\begin{itemize}
\item the set of all points of $P(V)$,
\item the set of all isotropic points of $P(V)$ with respect to the form $f$.
\end{itemize}
We also assume that $G \leq GL(V)$ is a transitive permutation group on $L$, $S$ is an $r-$partial spread in $V$ such that $\widetilde{S}$  partitions $L$,  $W \in S$ and $w \in P(W)$. Suppose there exist sets $A$, $B �\subseteq  G$ such that
\begin{enumerate}
\item $A$ acts sharply transitively on $S$ with respect to $W$ under the action of $G$ on the set of
all $r-$dimensional subspaces of $V$.

\item $B$ acts sharply transitively on $L \cap P(W)$ with respect to $w$ under the action of $G$ on
$P(W)$.
\end{enumerate}
Then $[A, B, G_w]$ is an $LS$ for $G$.
\end{thm}

Following Holmes \cite{9}, we assume that $G$ has an $MLS$ $\alpha = [A_1, \cdots, A_s]$, where $A_1$ is the sequence of all elements in a subgroup $H$ and $\sum_{i=2}^s|A_i|$ = $\sum_{i=1}^ka_ip_i$, where $[G:H]$ = $p^{a_1} \cdots p^{a_k}$. Then we say that $G$ has an $MLS$ over $H$. We write also $\alpha = [H, A_2, \cdots, A_s]$ to say that $G$ has an $MLS$ over the subgroup $H$.

\begin{thm} \label{thm4} {\rm (}See \cite[Lemma 1.1]{28}{\rm )}
The following  hold:
\begin{enumerate}
\item If $G$ has an $MLS$ over a subgroup $H$ then it has an $MLS$ over any conjugate of $H$,

\item Every finite nilpotent group has an $MLS$ over all of its subgroups,

\item Every finite solvable group has an $MLS$ over its normal and its Hall subgroups,

\item Every finite solvable group $G$ has an $MLS$ over all of its $p-$subgroups, where $p$ is a prime divisor of $|G|$.
\end{enumerate}
\end{thm}

A subset $Y$ of a group $G$ is called a cyclic set if there exists an
element $x \in G$ of order $\geq |Y|$ such that $Y = \{ x^i \ | \
0 \leq i \leq |Y|-1\}$. Suppose $\alpha = [A_1, \ldots, A_s]$ is an $MLS$ for  $G$. In such a case, the cyclic subset Y has an MLS \cite[Lemma 3.10]{32}. The $MLS$ $\alpha$
is called cyclic, if all of its blocks are cyclic subsets of
$G$.

Lempken and van Trung \cite[Theorem 4.1]{14} presented their
strong method for constructing an $MLS$ for a finite group $G$ by
using $MLS$ of its subgroups. We mention here the following new form of this result given by the present authors \cite[Theorem 2.1]{28}.

\begin{thm}\label{thm5}
Suppose $G$ is a finite group and
$H, K \leq G$ such that $H \cap gKg^{-1} = 1$, for all $g \in G$.
Suppose $G = \cup_{i=1}^n Hg_iK$ is the double coset
decomposition of $G$ with respect to $H$ and $K$ and $\{
g_i\}_{i=1}^n = \prod_{i=1}^mA_i$, where $A_i \subseteq G$ for $1
\leq i \leq m$ and $n = \prod_{i=1}^m|A_i|$. Moreover, we assume
that $H$ and $K$ have an $MLS$. The group $G$ has an $MLS$, if $n=1$ or for each $i$,
one of the following occurs:
\begin{itemize}
\item  $A_i$ is a cyclic subset,
\item $A_i$ is a subgroup of $G$ that has $MLS$,
\item $|A_i|$ is $4$ or a prime number.
\end{itemize}
\end{thm}

Holmes \cite{9} presented six very important conditions for existence an $MLS$ for a finite group. We mention below  one of conditions which is important in proving our main result.

\begin{thm} \label{thm6} {\rm (} Holmes \cite[Condition 2.4]{9}{\rm )} If $G$ contains two subgroups $H$ and $K$ such that
\begin{itemize}
\item $H$ has an $MLS$,
\item $K$ acts transitively in the permutation representation of $G$ on the cosets of $H$,
\item $K$ has an $MLS$ over $K \cap H$,
\end{itemize}
then $G$ has an $MLS$.
\end{thm}






We refer the interested readers to consult the interesting work of Liebeck et al. \cite{17}
 and the famous book of Dixon for other useful results related to our techniques \cite{3}.

\section{The Proof of \cite{11}}
The aim of this section is to show that the proof of existence of $MLS$ for the unitary groups $PSU_n(q),$ $SU_n(q)$, $PGU_n(q)$ and $GU_n(q)$ given by Hong et al. \cite{11} is not correct. As a consequence of our discussion, the $MLS$ conjecture for unitary groups is still open.

Singhi and Singhi \cite{31}, presented a strong method for construction of an $MLS$ based on parabolic subgroups and partitions obtained from classical spread in $P(V)$.

The process of proof in \cite{11} is as follows:

\begin{enumerate}
\item Construction of parabolic subgroup and proof of existence an $MLS$ for this subgroup. The structure of parabolic subgroups of unitary groups are considered into account in \cite{25,37}.

\item Applying the partition of an $r-$partial spread in the space of $P(V)$.

\item The proof of existence some special subgroups in unitary group by using some results in \cite{2}.
\end{enumerate}

In continue, we assume that $n$ is odd, $q = p^s$ is a prime power and $V = GF(q^{2n})$.  Consider $V$ as a vector space on dimension $n$ on $GF(q^2)$. Define the conjugate-symmetric sesquilinear form $f$ by $f(x,y)$ = $tr_{GF(q^{2n})/GF(q^2)}xy^q$ = $\sum_{i=0}^{n-1}(xy^q)^{q^{2i}}$.  Then the isometry group of $(V,f)$ will be isomorphic to $GU_n(q)$. Suppose $G = SU_n(q)$. The main proof of \cite{11} is based on Theorem \ref{thm3}. We apply this result in the case  that $G = SU_n(q)$ and $L$ is the set of all isotropic points of the inner product space $(V,f)$. The authors of \cite{11} used the partial spread $S_3$ = $\{ W^\prime_i \ | \ 0 \leq i \leq \frac{q^n+1}{q+1}\}$ and claimed that $S_3$ is classical. To prove this result, they assumed that $n = 2m+1$, $W^\prime = W^\prime_0 = GF(q^{2m})$ and $W^\prime_i = W^\prime_0\alpha^{i(q^m-1)}$.

Notice that $V=GF(q^{2n})$ is a vector space of dimension $2n$ on $GF(q)$ and $W = W_0 = GF(q^n)$ is simultaneously its subspace and its unique subfield of order $q^n$. In \cite{11}, it is claimed that $W^\prime_i$'s are totally isotropic subspace. In what follows,  we will show that this claim is incorrect.
\begin{enumerate}
\item Since $f(1,1)=tr_{\frac{GF(q^{2n})}{GF(q^2)}}1$ = $\sum_{i=0}^{n-1}1$ = $n$, $f(1,1) = 0$ implies that $p | n$. Hence $1 \in W^\prime$ and the point $1$ is not necessarily an isotropic point with respect to $f$.

\item Let $W^\prime = W_0^\prime = GF(q^{2m})$ be a totally isotropic subspace. Then for some $i,j$, $W^\prime_i \cap W^\prime_j \ne \langle 0 \rangle$ and $S_3$ is not necessarily a partition of all isotropic points of $P(V)$.

\item If $W = W^\prime_0 = GF(q^n)$ then $S_3$ is classical and if $W^\prime_0 $ is totally isotropic then for odd $n$, the subgroup $A$ = $\langle x^{(q^n-1)(q+1)} \rangle$ of order $\frac{q^n+1}{q+1}$, the  singer cycle subgroup of $GL_{2n}(q)$, is also a subgroup of $SU_n(q)$. Notice that $SU_n(q)$ acts transitive on the set of all isotropic points. On the other hand, $A$ is sharply transitive on the spread $\tilde{S_3}$ with respect to $P(W_0^\prime)$ then there must be  $\frac{(q^n+1)(q^{n}-1)}{(q+1)(q^2-1)}$ isotropic points  which contradicts by this fact that $P(V)$ has exactly $\frac{(q^n+1)(q^{n-1}-1)}{q^2-1}$ isotropic points.
\end{enumerate}

We now consider the existence problem of subgroups of $SU_n(q)$ that applied in \cite{11} for the construction of an  $MLS$. To do this, we assume that $$B = \{ e_1, e_2, \ldots, e_m, v, f_1, f_2, \ldots, f_m\}$$ is a basis for the vector space $V$ constructed in \cite[Proposition 1]{11}, where $$\{ e_1, e_2, \ldots, e_m\}$$ is a basis for the totally isotropic space $W$. To construct the subgroups $A$ and $B$ in \cite[Proposition 4]{11}, the authors used the  maximal tori subgroups of $SU_n(q)$ given in \cite{2}. If $n$ is odd then $SU_n(q)$ has a Singer cycle subgroup $A_3$ of order $\frac{q^n+1}{q+1}$. In \cite{11}, it is assume that $|A_3| = q^n + 1$ which is trivially incorrect. On the other hand, the group   $SU_n(q)$ has a cyclic subgroup of order $q^{n-1} - 1$ with generator $b_3$ such that
$$b_3 = \left[\begin{array}{cc} c & 0 \\  0 & (c^\ast)^{-1}
\end{array} \right],$$
where $c$ is a Singer cycle element of $GL_{\frac{n-1}{2}}(q^2)$. Suppose $B_3 = \langle b_3 \rangle$ and $D_3 = \langle b_3^{\frac{q^{n-1}-1}{q^2 -1}} \rangle$ is a subgroup of order $q^2 -1$ in $B_3$. Since $V$ is a vector space over $GF(q^{2})$, by \cite[Remark 6.1]{32} the left transversal $B^\prime_3 \in lt(B_3,D_3)$ of order $\frac{q^{n-1} - 1}{q^2 - 1}$ is sharply transitive on $P(W^\prime)$ with respect to $\langle e_1 \rangle$. Then the authors assumed that $|B^\prime_3| = \frac{q^{n-1} - 1}{q -1 }$  which is incorrect. As we mentioned, the number of isotropic points of $P(V)$ is equal to $\frac{(q^{n-1} - 1)(q^n+1)}{q^2-1}$ and so $A_3B_3^\prime$ which has order at most $\frac{(q^{n-1} - 1)(q^n+1)}{(q+1)(q^2-1)}$ cannot be sharply transitive on the set of all isotropic points of $P(V)$.

By \cite{2}, the mentioned tori subgroups are containing $Z(G)$ and the image of these subgroups under canonical homomorphism in the simple group $PSU_n(q)$ have orders $\frac{q^n+1}{d(q+1)}$ and $\frac{q^{n-1} - 1}{d}$, where $d = (n,q+1)$. Therefore, by using the product of the image of $A_3$ and $B_3^\prime$ under canonical homomorphism cannot construct a set which is sharply transitive on the set of all isotropic points of $P(V)$.

In this section, it is shown that the proof of Hong et al. \cite{11} in proving the existence of $MLS$ for unitary groups $PSU_{2n+1}(q)$ and $SU_{2n+1}(q)$ are incorrect. A similar argument shows that the proof of existence $MLS$ for $PSU_{2n}(q)$ and $SU_{2n}(q)$  are also incorrect. So, the following question is still open:

\begin{qu}
Is there an $MLS$ for simple unitary groups?
\end{qu}

\section{ Existence of $MLS$ for some Unitary Simple Groups}

The aim of this section is to prove the existence of $MLS$ in  some simple unitary groups. In the following theorem it is proved that if the group $SU_{2n-1}(q)$ has an $MLS$ then the projective special linear group $PSU_{2n}(q)$ has an $MLS$. As a consequence the problem of existence $MLS$ for the  projective special unitary group $PSU_{n}(q)$ will be reduced to the problem of existence for $PSU_{n}(q)$, when  $n$ is odd.

\begin{thm} \label{4.1}
Suppose $q$ is a prime power and $n \geq 2$. If $SU_{2n-1}(q)$ has an $MLS$ then  $PSU_{2n}(q)$ has an $MLS$.
\end{thm}

\begin{proof}
By \cite[Theorem 1.1, Proposition 7.8 and Lemma 7.10]{16}, $$PSU_{2n}(q) = \left( q^{n^2}:\frac{q^{2n-1}}{d(q+1)}\right)SU_{2n-1}(q),$$
where $d = (q+1,2n)$. Suppose $H = q^{n^2}:\frac{q^{2n-1}}{d(q+1)}$ and $K=SU_{2n-1}(q)$. Then $|H \cap K| | q^{(m-1)^2}$ and so it is a $p-$subgroup of $H$. Since $H$ is a product of a $p-$group and a cyclic group, it is solvable. Hence by Theorem \ref{thm4}, $H$ has an $MLS$ over $H \cap K$ and by Theorem  \ref{thm6}, the group $PSU_{2n}(q)$ has an $MLS$ if $SU_{2n-1}(q)$ has an $MLS$.
\end{proof}

\begin{thm} \label{4.2}
Suppose $q = 2^n$, $n > 1$. If $q + 1$ or $q^2 - q + 1$ is prime then the projective special unitary group $PSU_3(q)$ has an $MLS$.
\end{thm}

\begin{proof} Suppose $d = (3,q+1)$. Then the simple unitary group $PSU_3(q)$ has order $\frac{q^3(q^3+1)(q^2-1)}{d}$. Consider an automorphism $\tau$ of order $2$ in the field $G = GF(q^2)$ and $0 \ne k \in GF(q^2)$. Following
Fang et al. \cite{4}, we define:

\begin{center}
$k_\tau=\left(\begin{array}{lll} k^{\tau} & 0 & 0\\ 0 & k^{\tau -
1} & 0\\ 0 & 0 & k
\end{array}\right)$ and $Q(a,b) =\left(\begin{array}{lll} 1 & a & b\\ 0 & 1 & a^\tau\\ 0 & 0 &
1\end{array}\right)$,
\end{center}
where $a,b$ are elements of $GF(q^2)$ such that
\begin{equation}
a^{1+\tau}, b + b^\tau \in GF(q) \ \text{and} \ a^{1+\tau} + b + b^\tau = 0.
\end{equation}

Suppose $Q(q)$ is the subgroup generated by all matrices $Q(a,b)$ in which  $a, b$ satisfy the Equation 1 and $K(q) = \{ k_\tau \ | \ 0 \ne k \in GF(q^2)\} \cong Z_{\frac{q^2-1}{d}}$. By \cite[Chapter II, 10.12]{13}, $Q(q)$ is a $2-$Sylow subgroup of $PSU_3(q)$ with order $q^3$ and $N_{PSU_3(q)}(Q(q)) = Q(q):K(q)$. By a result in \cite{7},  each maximal subgroup of $PSU_3(q)$ is isomorphic to one
of the following groups:
\begin{enumerate}
\item $Q(q):K(q)$,

\item $Z_{\frac{q+1}{d}} \times PSL_2(q)$,

\item $\left(Z_{\frac{q+1}{d}} \times Z_{q+1}\right):S_3$,

\item $Z_{\frac{q^2-q+1}{d}}:Z_3$

\item $PSU_3(2^m)$, where $\frac{n}{m}$ is an odd prime greater than
$3$,

\item $PSU_3(2^m).Z_3$, where $\frac{n}{m}=2$ or $3$.
\end{enumerate}

\vskip 3mm

We first prove that $(q^3(q-1),(q+1)^2) = 1$
and if $q+1$ is prime then $$(q^3(q^2-1),q^2-q+1) = 1.$$ Notice
that the first part is a consequence of $(q^3,(q+1)^2) =
(q+1,q-1) = 1.$ To prove the second part, we assume that $a =
(q^2 - 1, q^2-q+1) > 1$ and $p$ is a prime factor of $a$.
Clearly, $(q-1,q+1) = 1$ and $p | q^2-1 = (q-1)(q+1)$. If $p | q
- 1$ then $p | q^2 - q + 1 + q - 1 = q^2$ and so $p = 2$, which
is impossible. If $p | q + 1$ then $p = q+1$. Since $p | q^2 - q
+ 1 = (q+1)(q-2) + 3$, $q=2$ which is our final contradiction.
Therefore, $a=1$ and $(q^3(q^2-1),q^2-q+1) = 1$.

Next we assume that $d = (3,q+1)$ and $q+1$ is a prime. Since $q \geq
4$, $d = 1$. From the structure of maximal subgroups of $G =
PSU_3(q)$, one can see that this group has a subgroup $H$ of
order $q^2 - q + 1$. Moreover, $G$ has another subgroup $F =
Q(q):K(q)$ of order $q^3(q^2-1)/d = q^3(q^2-1)$. By above argument, $|F|$ and
$|H|$ are coprime and so for each $g \in G$, $H \cap F^g = 1$. On
the other hand, $G = \bigcup_{i=1}^{q+1}Hg_iF$, for some $g_i \in
G$. Since $H$ and $F$ are solvable, by \cite[Proposition 3.1]{5},
they have $MLS$. Since $q+1$ is prime and $|PSU_3(q)| =
q^3(q+1)^2(q-1)(q^2-q+1)$, the subgroups $H$ and $F$ satisfies the
conditions of Theorem \ref{thm5} which implies that $G$ has an $MLS$.

Finally, we assume that $q^2-q+1$ is prime. Consider subgroups $H$
and $F$ of orders $\frac{(q+1)^2}{d}$ and $q^3(q-1)$,
respectively. By above argument, orders of $H$ and $F$ are
coprime and so for every $g\in G$, $H \cap F^g = 1$.  On the
other hand, $G = \bigcup_{i=1}^{q^2-q+1}Hg_iF$, for some $g_i \in
G$. Again $H$ and $F$ are solvable and so they have $MLS$. Since
$|PSU_3(q)| = q^3\frac{(q+1)^2}{d}(q-1)(q^2-q+1)$, the subgroups
$H$ and $F$ satisfies the conditions of Theorem \ref{thm5}
and $q^2-q+1$ is prime then simple unitary group $PSU_3(q)$ has
an $MLS$. This completes the second part  of our main theorem.
\end{proof}

\begin{thm} \label{4.3}
Suppose $q = 2^n$, $n > 1$. If $q + 1$ or $q^2 - q + 1$ is prime then the projective special unitary group $PSU_4(q)$ has an $MLS$.
\end{thm}

\begin{proof}
The proof follows from Theorems \ref{4.1} and \ref{4.2}.
\end{proof}

\begin{thm} \label{4.4}
Suppose $q$ is an odd prime power. Then,
\begin{enumerate}
\item If $q^2 - q + 1$ is prime then $PSU_3(q)$ has an $MLS$.

\item If $q > 5$ and $q + 1 = 2p$, $p$ is prime, then $PSU_3(q)$ has an $MLS$.
\end{enumerate}
\end{thm}

\begin{proof}
Let $\Omega$  be the set all isotropic points of $PG(2,q)$ of size $q^3+1$ and consider the
permutation representation of $PSU_3(q)$ on $\Omega$ \cite{25}. Then the stabilizer of an isotropic point $\alpha$ in $PSU_3(q)$ has the form $S_\alpha = q.q^2:\frac{q^2-1}{d}$, where $d = (3,q+1)$. Since $d|q+1$, $S_\alpha$ has a subgroup of order $q - 1$. On the other hand, the group $S_\alpha$ has a normal subgroup of order $q^3$ and so we can construct a subgroup $H$ of $S_\alpha$ of order $q^3(q-1)$. By \cite{26}, $PSU_3(q)$ has a subgroup isomorphic to $K = Z_{q+1} \times Z_{\frac{q+1}{d}}$ of order $\frac{(q+1)^2}{d}$. Since $(|H|, |K|) = 1$, the subgroups $H$ and $K$ satisfies conditions of Theorem \ref{thm5}.  The subgroups $H$ and $K$ are solvable and so they have  $MLS$. Therefore, if  $q^2 - q + 1$ is prime then by Theorem \ref{thm5}, the projective special unitary group $PSU_3(q)$ has an $MLS$.

We now assume that $q > 5$ and $q + 1 = 2p$, where $p$ is prime. This shows that $p > 3$ and $d = 1$. Again, we consider the permutation representation of $PSU_3(q)$ on  $\Omega$. Let $H=S_\alpha$, the stabilizer of an isotropic point $\alpha$, and Choose $K$ to be Singer cycle subgroup of order $\frac{q^2-q+1}{d} = q^2 - q + 1$. Then $H$ and $K$ satisfies the conditions of Theorem \ref{thm5}. Since $q + 1 = 2p$ and by \cite{34}, the group $PSU_3(q)$ is $2-$transitive on $\Omega$.  Define $A_1 = \{ x_1,x_2\}$, where $x_1, x_2 \in PSU_3(q)$,  $\alpha^{x_1} = \beta_1$ and $\alpha^{x_2} = \beta_2$. The subgroup $K$ has $q+1$ orbits. Since $PSU_3(q)$ is $2-$transitive, if $\Delta_1$ and $\Delta_2$ are a pair of orbits of $K$, then there exists an element $y \in PSU_3(q)$ such that $\beta_1^y \in \Delta_1$ and $\beta_2^y \in \Delta_2$. Consider orbits of $K$ as $p$ different pairs then set $A_2 = \{ y_1, \ldots, y_p\}$ that $y_i\in PSU_3(q)$ is corresponding to $i$-th pair of orbits for $1 \leq i \leq p$. It is clear that $H$ and $K$ are solvable and so they have $MLS$. Finally, since the size of $A_1$ and $A_2$ are prime,  by Theorem \ref{thm5}, the unitary group $PSU_3(q)$ has $MLS$.
\end{proof}

\begin{thm}
Suppose $q$ is an odd prime power. Then,
\begin{itemize}
\item If $q^2 - q + 1$ is prime then $PSU_4(q)$ has an $MLS$.

\item If $q > 5$ and $q + 1 = 2p$, $p$ is prime, then $PSU_4(q)$ has an $MLS$.
\end{itemize}
\end{thm}

\begin{proof}
The proof follows from Theorems \ref{4.1} and  \ref{4.4}.
\end{proof}

\section{Concluding Remarks}
The present authors in \cite{28} applied the results of \cite{11} to prove that all finite simple groups of orders $\leq 10^{12}$ have $MLS$ other than the Ree group $Ree(27)$, the O'Nan group $O'N$ and the untwisted group  $G_2(7)$. There are $12$ unitary groups of orders $\leq 10^{12}$. These are $PSU_3(9)$, $PSU_3(13),$ $PSU_3(17)$, $PSU_4(5),$ $PSU_3(19),$ $PSU_3(23),$ $PSU_3(25)$, $PSU_3(29)$, $PSU_5(3)$, $PSU_3(27)$,
$PSU_3(32)$ and $PSU_3(31)$. In this paper, it is shown that the proof of the main result of \cite{11} regarding to existence of an $MLS$ for projective special unitary groups is not correct and so the existence of an $MLS$ for these groups have to be checked.

In \cite{5,14}, the existence of an $MLS$ for the unitary simple
groups $PSU_3(2^2)$  and $PSU_3(2^4)$ were proved. By Theorem \ref{4.2}, simple unitary groups  $PSU_3(2^8)$, $PSU_3(2^{16})$ and $PSU_3(2^{32})$ have $MLS$. By \cite{5}, the simple unitary group $PSU_3(5)$ has $MLS$. On the other hand, by Theorem \ref{4.1} the unitary group $PSU_4(5)$ and by Theorem \ref{4.3}(2) the simple groups $PSU_3(9)$, $PSU_3(13)$ and $PSU_3 (25)$ have $MLS$. Therefore, the existence of an $MLS$ for eight simple unitary groups of order $\leq 10^{12}$ is not still solved.

\vskip 3mm

\noindent\textbf{Acknowledgement.}  The research of the  authors are partially supported by INSF under grant number 93010006.


\begin{thebibliography}{33}

\bibitem{2} L. Babai, P. P. P$\grave{\rm a}$lfy and J. Saxl, On the number of $p-$regular elements in finite simple groups, {\it LMS J.
Comput. Math.} {\bf 12} (2009) 82--119.

\bibitem{3} J. D. Dixon and B. Mortimer, \textit{Permutation Groups}, Graduate Texts in Mathematics \textbf{163}, Springer-Verlag, New York, 1996.

\bibitem{4} X. G. Fang, G. Havas and J. Wang, A family of non$-$quasiprimitive graphs admitting a quasiprimitive
$2-$arc transitive group action, {\it Eur. J. Combin.} {\bf 20}
(1999) 551--557.

\bibitem{5} M. I. Gonz$\acute{\rm a}$lez Vasco, M. R$\ddot{\rm o}$tteler and R. Steinwandt, On minimal length factorizations of finite groups,
{\it Exp. Math.} {\bf 12} (1) (2003) 1--12.

\bibitem{6} M. I. Gonz$\acute{\rm a}$lez Vasco and R. Steinwandt, Obstacles in two public key cryptosystems
based on group factorizations, {\it Tatra Mt. Math. Publ.} {\bf
25} (2002) 23--37.

\bibitem{7} R. W. Hartley, Determination of the ternary collineation groups whose coefficients lie in the $GF(2^n)$, {\it Ann. Math.} {\bf 27} (1926) 140--158.

\bibitem{8} M. D. Hestenes, Singer groups, \textit{Canad. J. Math.}  \textbf{22} (1970) 492--513.

\bibitem{9} P. E. Holmes, On minimal factorisations of sporadic groups, \textit{ Exp. Math.} {\bf 13}(4) (2004) 435--440.

\bibitem{10} H. Hong, L. Wang, Y. Yang and H. Ahmad, All exceptional groups of Lie type have minimal logarithmic signatures, {\it Appl. Algebra Engrg. Comm. Comput.} \textbf{25} (4) (2014) 287--296.


\bibitem{11} H. Hong, L. Wang, Y. Yang, Minimal logarithmic signatures for the unitary
group $U_n(q)$, \textit{Des. Codes Cryptogr.} \textbf{77} (1) (2015) 179--191.

\bibitem{12}  B. Huppert, Singer-Zyklen in klassischen Gruppen, \textit{Math. Z.} \textbf{117} (1970) 141--150.

\bibitem{13} B. Huppert, {\it Endliche Gruppen I}, Springer-Verlag, Berlin, 1967.

\bibitem{14} W. Lempken and T. van Trung, On minimal logarithmic signatures of finite groups, {\it Exp. Math.} {\bf 14} (3) (2005) 257--269.

\bibitem{15} W. Lempken, T. van Trung, S. S. Magliveras and W. Wei,
A public key cryptosystem based on non-abelian finite groups,
{\it J. Cryptol.} {\bf 22} (2009) 62--74.

\bibitem{16} C. H. Li and B. Xia, Factorizations of almost simple groups with a solvable factor, arXiv:1408.0350v3.

\bibitem{17} M. W. Liebeck, C. E. Praeger and J. Saxl, The maximal factorizations of the finite simple groups and their authomorphism groups, \textit{Mem. Amer. Math. Soc.} {\bf 86} (432) (1990), iv+151 pp.

\bibitem{19} S. S. Magliveras, B. A. Oberg and A. J. Surkan, A new random
number generator from permutation groups, {\it Rendiconti del
Seminario Matematico di Milano} {\bf 54} (1985) 203--223.

\bibitem{20} S. S. Magliveras, A cryptosystem from logarithmic signatures of
finite groups, In Proceedings of the 29th Midwest Symposium on
Circuits and Systems, pp. 972--975, Elsevier Publishing Company,
Amsterdam, 1986.

\bibitem{21} S. S. Magliveras and N. D. Memon, Properties of cryptosystem PGM,
in {\it Advances in Cryptology: Crypto '89}, Lecture Notes in
Computer Science {\bf 435}, Springer-Verlag, Berlin (1989), pp.
447--460.

\bibitem{22} S. S. Magliveras and N. D. Memon, Complexity tests for
cryptosystem PGM, {\it Congressus Numerantium} {\bf 79} (1990)
61--68.

\bibitem{23} S. S. Magliveras and N. D. Memon, Algebraic
properties of cryptosystem PGM, {\it J. Cryptol.} {\bf 5} (1992)
167--183.

\bibitem{24} S. S. Magliveras, D. R. Stinson and T. van Trung,  New approaches to designing public key cryptosystems using one-way functions and trapdoors in finite groups, \textit{ J. Cryptology} \textbf{15} (4) (2002) 285--297.


\bibitem{25} V. D. Mazurov, Minimal permutation representations of finite simple classical groups, Special linear, symplectic, and unitary groups, \textit{ Algebra and Logic} \textbf{32} (3) (1993) 142--153.

\bibitem{26} H. H. Mitchell, Determination of the ordinary and modular ternary linear groups, \textit{Trans. Amer. Math. Soc.} \textbf{12} (2) (1911) 207--242.

\bibitem{27} A. R. Rahimipour, A. R. Ashrafi and A. Gholami, The existence of minimal logarithmic signatures for the sporadic Suzuki and simple Suzuki groups, \textit{Cryptogr. Commun.} \textbf{7} (4) (2015) 535--542.

\bibitem{28} A. R. Rahimipour, A. R. Ashrafi and A. Gholami, The existence of minimal logarithmic signatures  for some finite simple groups, \textit{Exp. Math.} {\bf 27} (2) (2018) 138--146.



\bibitem{31} N. Singhi and N. Singhi, Minimal logarithmic signatures for classical groups, {\it Des. Codes Cryptogr.} {\bf 60} (2) (2011) 183--195.

\bibitem{32} N. Singhi, N.  Singhi and S. Magliveras, Minimal logarithmic signatures for finite groups of Lie type, \textit{ Des. Codes Cryptogr.} {\bf 55} (2-3) (2010) 243--260.

\bibitem{34} M. Suzuki, A characterization of the $3-$dimensional projective unitary group over a finite field of odd characteristic, \textit{J. Algebra} \textbf{2} (1965) l--14.

\bibitem{35} P. Svaba, T. van Trung and P. Wolf, Logarithmic signatures for abelian groups and their
factorization, {\it Tatra Mt. Math. Publ.} {\bf 57} (2013) 21--33.

\bibitem{37} R. A. Wilson, \textit{The Finite Simple Groups},  Graduate Texts in Mathematics \textbf{251}, Springer-Verlag London Ltd., London, 2009.

\end{thebibliography}
\end{document}